\documentclass{amsart}
\usepackage{amssymb}
\usepackage{latexsym}
\usepackage{graphicx}
\usepackage{hyperref}
\usepackage{comment}

\newcommand{\Z}{{\mathbb Z}}
\newcommand{\R}{{\mathbb R}}
\newcommand{\C}{{\mathbb C}}

\newtheorem{lemma}{Lemma}[section]
\newtheorem{theorem}[lemma]{Theorem}

\newtheorem{corollary}[lemma]{Corollary}

\newcommand{\nn}{\nonumber}
\newcommand{\be}{\begin{equation}}
\newcommand{\ee}{\end{equation}}

\newcommand{\ol}{\overline}
\newcommand{\ti}{\tilde}

\newcommand{\spr}[2]{\langle #1 , #2 \rangle}

\newcommand{\tr}{\mathrm{tr}}
\newcommand{\im}{\mathrm{Im}}

\DeclareMathOperator{\dist}{dist}

\newcommand{\eps}{\varepsilon}

\numberwithin{equation}{section}

\begin{document}

\title{Probabilistic averages of Jacobi operators}

\author[H.\ Kr\"uger]{Helge Kr\"uger}
\address{Department of Mathematics, Rice University, Houston, TX~77005, USA}
\email{helge.krueger@rice.edu}

\thanks{H.\ K.\ was supported by NSF grant DMS--0800100.}

\date{\today}


\begin{abstract}
 I study the Lyapunov exponent and the integrated density of states
 for general Jacobi operators. The main result is that questions about
 these, can be reduced to questions about ergodic Jacobi operators.
 Then, I apply this to $a(n) = 1$ and $b(n) = f(n^\rho \pmod{1})$
 for $\rho > 0$ not an integer, and to obtain a probabilistic
 version of the Denisov--Rakhmanov--Remling Theorem.
\end{abstract}

\maketitle

\section{Introduction}

This paper is part of my effort to study the Schr\"odinger operator,
\be\label{eq:Hrho}
 (H u)(n) = u(n+1) + u(n-1) + f(n^\rho \pmod 1) u(n),
\ee
where $u(-1) = 0$, $f: [0,1]\to\R$ is a continuous function,
and $\rho > 0$ is not an integer. I was intrigued by the
fact that for $0 < \rho < 1$ and $f(0) \neq f(1)$, one has
the absence of absolutely continuous spectrum and vanishing
of the Lyapunov exponent on an interval. This is somewhat
surprising since $n^\rho \pmod 1$ has nice uniform distribution
properties. We will discuss properties of these operators
in Section~\ref{sec:rho}. In particular, we resolve the discrepancy
between the perturbative and numerical calculations of Griniasty
and Fishman in \cite{gf} in Corollary~\ref{cor:lyaprho} by
proving an exact formula.

In order to understand the consequences and reasons for zero
Lyapunov exponent, it turned out to be useful to work with
general Jacobi operators, which are introduced by
\be
 \begin{split}
  J:\ell^2(\Z) & \to \ell^2(\Z) \\
  Ju (n) &=a(n) u(n+1) + b(n) u(n) + a(n-1) u(n-1),
 \end{split}
\ee
where $C_0^{-1} \leq a(n) \leq C_0$ and $-C_0 \leq b(n) \leq C_0$
for some $C_0 > 1$. We let $m_{\pm}(z)$ be the
Weyl--Titchmarsh $m$ functions of the restrictions of
$J$ to $\ell^2(\Z_{\pm})$. $J$ is called reflectionless on $A$
if
\be
 m_+(t) = - \ol{m_-(t)}
\ee
for almost every $t \in A$. Denote by $L(E)$ the Lyapunov
exponent of $J$, by $J^{(n)}$ the $n$-th translate of $J$,
and by $\delta_{J}$ the Dirac measure. We have that

\begin{theorem}\label{thm:int1}
 Assume $L(E) = 0$ for almost every $E \in A$ and
 \be
  \mu = \lim_{N\to\infty} \frac{1}{N} \sum_{n=0}^{N-1} \delta_{J^{(n)}}
 \ee
 in the weak $\ast$ topology, then $\mu$ almost every
 Jacobi operator has absolutely continuous spectrum in
 the essential closure of $A$ and is reflectionless
 there.
\end{theorem}

Remling has shown in \cite{rem} a very similar result.
He has assumed that $A \subseteq \sigma_{\mathrm{ac}}(J)$,
and concluded that every $J$ in the $\omega$ limit set of $J^{(n)}$
is reflectionless on $A$. Since
\be
 \sigma_{\mathrm{ac}}(J) \subseteq \{E:\quad L(E) = 0\},
\ee
the assumptions of the above theorem are weaker, but also
the conclusion is. One can easily
check that sparse potentials (as discussed in \cite{rem}) provide
examples, that show that this distinction is sharp.

The above theorem will follow Theorem~\ref{thm:main}, which
provides a formula for the Lyapunov exponent $L(E)$ in terms
of the Lyapunov exponents of the ergodic families arising
in the ergodic decomposition of the limit measure $\mu$.

Theorem~\ref{thm:int1} implies in particular the following result, which has
to be thought of as a probabilistic analog of the 
Densiov--Rakhmanov--Remling theorem

\begin{theorem}
 Let $J$ be a Jacobi matrix with $\sigma_{\mathrm{ess}}(J) = [-2,2]$
 and $L(E) = 0$ for almost every $E \in [-2,2]$, then for every $\eps > 0$
 \be
  \lim_{N\to\infty} \frac{1}{N}\#\{1 \leq n \leq N:\quad |a(n) - 1|  > \eps\text{ or } |b(n)| > \eps\}
   = 0.
 \ee
\end{theorem}

We will obtain a generalization of this theorem for finite gap
operators (see Corollary~\ref{cor:drr}) and the above claim for the Lyapunov exponent
as corollaries of Theorem~\ref{thm:main} in the next section.

In Section~\ref{sec:space}, I collect a few results on the
space of all Jacobi operators and discuss measures on that
space. In Section~\ref{sec:ergodic}, I discuss results
about ergodic Schr\"odinger operators. The main results
are stated in Section~\ref{sec:results}. The application
to the potential $V(n) = f(n^\rho \pmod 1)$ is examined
in Section~\ref{sec:rho}. Section~\ref{sec:lyap} proofs
some facts about the Lyapunov exponent for general Jacobi matrices
and provides the proof of Theorem~\ref{thm:main}, which
has to be considered the main result of this paper.

%
%
%

\section{Probabilistic averages of Jacobi matrices}
\label{sec:space}

Given bounded sequences $a: \Z \to (0,\infty)$, $b:\Z\to\R$,
we introduce the associated Jacobi operator $J: \ell^2(\Z) \to \ell^2(\Z)$ by
\be
 (J u)(n) = a(n) u(n+1) + b(n) u(n) + a(n-1) u(n-1).
\ee
We will often identify $J$ with $(a,b)$. Fix now $C_0 > 1$,
and introduce $\mathcal{J}$ as the set of all Jacobi operators,
such that $(a,b)$ satisfy the inequalities
\be
 \frac{1}{C_0} \leq a(n) \leq C_0,\quad -C_0\leq b(n) \leq C_0.
\ee
We endow $\mathcal{J}$ with the strong operator topology,
which just corresponds to pointwise convergence on the level
of the sequences $(a,b)$. We remark that $\mathcal{J}$ is now
a compact metric space, where an explicit example of the metric
is
\begin{align}
 \nn d(J, \ti{J}) &= \sum_{n\in\Z}\frac{1}{2^{|n|}}(|\spr{\delta_n}{(J - \ti{J}) \delta_n}| + |\spr{\delta_n}{(J - \ti{J}) \delta_{n+1}}|) \\
 &= \sum_{n\in\Z} \frac{1}{2^{|n|}} (|b(n) - \ti{b}(n)| + |a(n) - \ti{a}(n)|).
\end{align}
Denote by $S$ the shift operator on $\ell^2(\Z)$ that is
\be
 (S u)(n) = u(n+1).
\ee
Introduce $\hat{S}: \mathcal{J} \to \mathcal{J}$ by
\be
 \hat{S} J = S^{\ast} J S,
\ee
and denote $J^{(n)} = \hat{S}^n J$.

We will denote by $\mathcal{M}^1$ the space of all Borel probability
measures on $\mathcal{J}$. For a Jacobi matrix $J$, we introduce
the corresponding Dirac measure $\delta_J$ by
\be
 \delta_{J}(A) = \begin{cases} 1 & J \in A \\ 0 & J \notin A. \end{cases}
\ee
We will be mainly interested in the limit points of the
averages of the Dirac measures of the translates of a Jacobi matrix.
For this, introduce for a Jacobi matrix $J$ and an integer $N\geq 1$
the average
\be
 A_{N,J} = \frac{1}{N} \sum_{n=0}^{N-1} \delta_{J^{(n)}},
\ee
which will be a measure in $\mathcal{M}^1$.

We denote by $\omega(J)$ the (topological) $\omega$ limit set of the translates
of $J$, that is
\be
 \omega(J) = \{\ti{J} \in\mathcal{J}:\quad \exists n_j\to\infty:\quad \ti{J} = \lim_{j\to\infty} J^{(n_j)}\}.
\ee
We recall that the weak $\ast$ topology on $\mathcal{M}^1$ gives
rise to the notion of convergence $\mu_n \to \mu$ if for every
continuous function $f: \mathcal{J} \to \R$ we have
\be
 \lim_{n\to\infty} \int f(J) \mu_n(J) = \int f(J) \mu(J).
\ee
We remark

\begin{lemma}
 $\mathcal{M}^1$ is a compact and metrizable space in the weak $\ast$ topology.
\end{lemma}

We write $\mathrm{supp}(\mu)$ for the support of a measure $\mu$, which
is the smallest closed set $A$, such that $\mu(A) = 1$. Furthermore,
we call a measure $\mu \in \mathcal{M}^1$
shift invariant, if for any Borel set $A\subseteq \mathcal{J}$
\be
 \mu(A) = \mu(\hat{S} A).
\ee

\begin{lemma}
 If $\mu = \lim_{j\to\infty} A_{N_j,J}$ in the weak $\ast$ topology,
 then
 \be
  \mathrm{supp}(\mu) \subseteq \omega(J),
 \ee
 and $\mu$ is shift invariant.
\end{lemma}

This implies the following consequence

\begin{lemma}\label{lem:sigess}
 Let $\mu = \lim_{j\to\infty} A_{N_j}(J)$ for some $N_j\to\infty$.
 Then for $\mu$ almost every $\ti{J}$, we have that
 \be
  \sigma_{\mathrm{ess}}(\ti{J}) \subseteq \sigma_{\mathrm{ess}}(J).
 \ee
\end{lemma}

\begin{proof}
 Follows from the fact, that the inclusion holds for every $\ti{J} \in \omega(J)$.
\end{proof}

We will need the following result from measure theory, known as Portmanteau-Theorem
(see e.g. \cite[Theorem~VIII.4.10.]{els}).

\begin{theorem}
 $\mu_n \to \mu$ in the weak $\ast$ topology, is equivalent to that for every Borel
 set $B$ with $\mu(\partial B)  =0$, we have that
 \be
  \lim_{n\to\infty} \mu_n(B) = \mu.
 \ee
\end{theorem}

We say that a sequence $J^{(n)}$ converges to a set $A \subseteq \mathcal{J}$
along $N_j \to \infty$ in probability if for every $\eps > 0$
\be
 \lim_{j \to \infty} \frac{1}{N_j} \#\{1 \leq n \leq N_j:\quad d(J^{(n)}, A) \geq \eps\} = 0.
\ee
We have the following result.

\begin{lemma}\label{lem:convtosupp}
 Let
 \be
  \mu = \lim_{j\to\infty} A_{N_j,J},
 \ee
 and $S$ a support for $\mu$. Then $J^{(n)}$ converge to $S$ along
 $N_j$ in probability.
\end{lemma}

\begin{proof}
 For $\eps > 0$ apply the last theorem to $B = \{J: \quad \dist(S, J) \geq \eps\}$
 to conclude that $A_{N_j,J}(B) \to 0$ as $j\to\infty$. By rewriting, one sees
 that this is exactly the definition of convergence in probability.
\end{proof}

For $\Lambda\subseteq\Z$, denote by $J_{\Lambda}$ the restriction of $J$ to $\ell^2(\Lambda)$.
We will call $f: \mathcal{J} \to \R$ compactly supported, if
there is a finite set $\Lambda \subseteq \Z$ such that
\be
 f(J) = f(\ti{J}),
\ee
whenever $J_{\Lambda} = \ti{J}_{\Lambda}$. We have that

\begin{lemma}\label{lem:compsupp}
 For $\mu_n,\mu \in \mathcal{M}^1$, we have
 \be
  \mu_n \to \mu
 \ee
 in the weak $\ast$ topology, if and only if for every
 compactly supported $f$ 
 \be\label{eq:convcomp}
  \lim_{n\to\infty} \int f d\mu_n = \int f d\mu.
 \ee
\end{lemma}

\begin{proof}
 It clearly suffices to show that \eqref{eq:convcomp} implies
 weak $\ast$ convergence. So given $f: \mathcal{J} \to \R$
 and $\eps > 0$, we have that show that there is an $N\geq 1$
 $$
  \forall n \geq N:\quad |\int f d\mu_n - \int f d\mu| \leq \eps.
 $$
 Since, $f$ is continuous, we may find for each $J \in \mathcal{J}$
 an integer $K \geq 1$ such that $|f(J) - f(\ti{J})| \leq \frac{\eps}{2}$,
 where
 $$
  \ti{J} \in U_{J, K} = \{\hat{J}:\quad \hat{J}_{[-K,K]} = J_{[-K,K]}\}.
 $$
 Since the $U_{J,K}$ are open sets and $\mathcal{J}$ is compact,
 finitely many of them cover $\mathcal{J}$. In particular, we can
 find a maximal necessary $K$. Hence, we may approximate $f$
 by $\ti{f}$, which is supported on $[-K,K]$ such that
 $\|f - \ti{f}\|_{\infty} < \frac{\eps}{2}$. Now the
 claim follows by \eqref{eq:convcomp}.
\end{proof}

%
%
%

\section{Families of ergodic Schr\"odinger operators}
\label{sec:ergodic}

In this section, we collect basic facts about ergodic
Jacobi operators. For the Jacobi operator background see \cite{d1} or Section~7 of \cite{siem}.
For the measure theoretic part, see \cite{kh} or \cite{phe}

Denote by $\mathcal{M}^1_S$ the set of all shift invariant measures. One can check that
$\mathcal{M}^1_S$ will be a convex set, and in particular we will
write $\mathcal{E}$ for its extremal points. 

It is a known fact in ergodic theory, that $\mathcal{E}$ are exactly
the ergodic measures of the dynamical system $(\mathcal{J}, \hat{S})$,
where
\be
 \hat{S}(J) = S^{\ast} J S.
\ee
Furthermore, it follows from Choquet's theorem that one can
write any measure $\mu \in \mathcal{M}^1_S$ as a generalized convex
combination. That is, there exists a measure $\alpha$ on $\mathcal{E}$
such that for any $f: \mathcal{J} \to \R$ continuous, one has
\be\label{eq:edc}
 \int f d\mu = \int \left(\int f d\beta\right) d\alpha(\beta).
\ee

We call an ergodic measure $\beta\in\mathcal{E}$ on the
space of Jacobi operators $\mathcal{J}$ a family  of ergodic Jacobi operators.
One knows from general fact, that there is a set $\Sigma(\beta)$
such that
\be
 \Sigma(\beta) = \sigma(J)
\ee
for $\beta$ almost every $J$. We may define its Lyapunov exponent by
\be
 \gamma_{\beta}(z) = \lim_{N \to \infty} \frac{1}{N} \int_{\mathcal{J}} 
 \log\left\| \prod_{n=N-1}^{0} \frac{1}{a(n)} \begin{pmatrix} z - b(n) & 1 \\ a(n)^2 &  0 \end{pmatrix} \right\| d\beta(J),
\ee
and its integrated density of states by
\be
 k_{\beta} (E) = \int_{\mathcal{J}} \spr{\delta_0}{\chi_{(-\infty,E)}(J) \delta_0} d\beta(J).
\ee
We note that this quantity is equal to
\be
 k_{\beta}(E) = \lim_{N\to\infty} \frac{1}{N} \int_{\mathcal{J}} \tr(P_{(-\infty,E)}(J_{[0,N-1]})) d\beta(J).
\ee
We will need the following result of Kotani theory. 

\begin{theorem}\label{thm:kot}
 Denote by $\mathcal{Z}$ the essential closure of the set
 \be
 \{E\in \R:\quad \gamma_{\beta}(E) = 0\}
 \ee
 then $\beta$ almost every $J$ has purely absolutely
 continuous spectrum on $\mathcal{Z}$ and it is reflectionless there.
\end{theorem}

We define
\be
 \log(A_{\beta}) = \int_{\mathcal{J}} \log(a(0)) d\beta(J).
\ee
Furthermore $m_+(z,J) = \spr{\delta_0}{(J_+ - z)^{-1} \delta_0}$,
where $J_+$ is the restriction of $J$ to $\ell^2(\Z_+)$.

\begin{lemma}\label{lem:ergodic}
 We have that
 \begin{align}
  \label{eq:egamma}   \gamma_{\beta}(z) & = \log(A_{\beta}^{-1}) - \int_{\mathcal{P}} \log|m_+(z,J)| d\beta(J) \\
  \label{eq:thouless}           & = \log(A_{\beta}^{-1}) + \int \log|t - z| d\nu_{\beta}(z) 
 \end{align}
 for every $z\in\C$.
\end{lemma}

\eqref{eq:thouless} is known as the \textit{Thouless formula}. It implies that
\be\label{eq:limegamma}
 \lim_{\eps\to 0} \gamma_{\beta}(E + i \eps) = \gamma_{\beta}(E)
\ee
for every $E\in\R$ by monotone convergence.

We now make the connection, to the usual definition of ergodic
Jacobi operators (see also Section~2 in \cite{cdk}).
Let $(\Omega, T, \mu)$ be an ergodic
dynamical system, and $a: \Omega \to (0,\infty)$ and $b:\Omega\to\R$
are measurable functions satisfying
\be
 \frac{1}{C_0} \leq a(\omega) \leq C_0,\quad -C_0 \leq b(\omega) \leq C_0
\ee
for almost every $\omega$. Then we can define a map
\be
 f: \Omega \to \mathcal{J}
\ee
by $f(\omega)$ being the Jacobi operator with coefficients
\be
 \{(a(T^n \omega), b(T^n \omega)\}_{n\in\Z}.
\ee
Introduce a measure $\beta$ on $\mathcal{J}$ given by
\be
 \beta(A) = \mu(f^{-1}(A))
\ee
for Borel subsets $A \subseteq \mathcal{J}$. Then the usual
definitions of the Lyapunov exponent and the integrated
density of states will just be $\gamma_\beta$ and $k_{\beta}$.

%
%
%

\section{Statement of the results}
\label{sec:results}

For a Jacobi operator $J \in \mathcal{J}$ and a sequence $N_j \to \infty$,
we introduce the Lyapunov exponent by
\be
 \ol{L}(z,J,\{N_j\}) = \limsup_{j\to\infty} \frac{1}{N_j} \log\left\|\prod_{n=N_j-1}^{0} \frac{1}{a(n)} \begin{pmatrix} z - b(n) & 1 \\ a(n)^2 & 0 \end{pmatrix} \right\|,
\ee
where $J = (a,b)$, and the integrated density of states by
\be
 k(E, J, \{N_j\}) = \lim_{j\to\infty} \frac{1}{N_j} \tr(P_{(-\infty,E)}(J_{[0,N_j-1]})),
\ee
where the limit is assumed to exist.

\begin{theorem}\label{thm:main}
 Assume that we have a measure $\alpha$ on $\mathcal{E}$ such that for $N_j \to \infty$
 \be\label{eq:asmulimit}
  \lim_{j\to\infty} \frac{1}{N_j} \sum_{n=0}^{N_j-1} \delta_{J^{(n)}} = \int_{\mathcal{E}} \beta d\alpha(\beta)
 \ee
 then
 \be\label{eq:olL}
  \ol{L}(z,J,\{N_j\}) = \int_{\mathcal{E}} \gamma_{\beta} (z) d\alpha(\beta)
 \ee
 for $\im(z) > 0$ and almost every $z\in\R$, and
 \be
  k(E,J,\{N_j\}) = \int_{\mathcal{E}} k_{\beta} (E) d\alpha(\beta).
 \ee
\end{theorem}

The statement is actually stronger, since one may replace the almost every
by quasi-every (in the sense of potential theory). Theorem~\ref{thm:int1} is now an easy corollary.

\begin{corollary}
 Let $A \subseteq \R$ be a set of positive measure. Assume that
 \be
  \ol{L}(E,J,\{N_j\}) = 0
 \ee
 for almost every $E \in A$. Then $\mu$ almost every $J$ is
 reflectionless on $A$.
\end{corollary}

\begin{proof}
 Since $\gamma_{\beta}(E) \geq 0$ for every $\beta$, we can conclude by
 \eqref{eq:olL} that for $\alpha$ almost every $\beta$, we have
 $$
  \gamma_{\beta}(E) = 0
 $$
 for almost every $E \in A$. The result now follows by Theorem~\ref{thm:kot}.
\end{proof}

We also obtain the following probabilistic version of
the Denisov--Rakhmanov--Remling theorem (see \cite{den}, \cite{rak}, \cite{rem}). For this recall, that a set
$\mathfrak{e}$ is called a finite gap set, if
\be
 \mathfrak{e} = [E_0, E_1] \cup [E_2, E_3] \cup \dots \cup [E_{2g+1}, E_{2g+2}].
\ee
Furthermore, one has a finite dimensional torus $\mathcal{T}(\mathfrak{e})$
of reflectionless Jacobi-operators which have spectrum $\mathfrak{e}$.
We have that

\begin{corollary}\label{cor:drr}
 Let $\mathfrak{e}$ be a finite gap set, and assume that
 \be
  \sigma_{ess} (J) = \mathfrak{e}
 \ee
 and $\ol{L}(E, \{N_j\},J) = 0$ for almost every $E \in \mathfrak{e}$,
 then
 \be
  J^{(n)} \to \mathcal{T}(\mathfrak{e})
 \ee
 in probability along $N_j$, where $\mathcal{T}(\mathfrak{e})$ denotes the isospectral
 torus.
\end{corollary}

\begin{proof}
 Assume there is a subsequence of $N_j$ such that convergence
 in probability does not hold.  By passing to a  further subsequence of $N_j$,
 we may assume that \eqref{eq:asmulimit} holds.
 Since $\gamma_{\beta}(z) \geq 0$ everywhere, it follows from
 our result that for $\alpha$ almost every $\beta$, we have that
 $$
  \Sigma(\beta) = \mathfrak{e}
 $$
 and that $\gamma_{\beta} = 0$ on $\mathfrak{e}$. Hence, Kotani's theory
 implies that these operators are reflectionless on $\mathfrak{e}$,
 which implies in turn that $\beta$ is supported on $\mathcal{T}(\mathfrak{e})$
 (see e.g. Section~8 in \cite{tjac}).
 This is a contradiction by Lemma~\ref{lem:convtosupp}.
\end{proof}

%
%
%
%

\section{The family of potentials}
\label{sec:rho}

In this section, we examine the family of Schr\"odinger operators
given by \eqref{eq:Hrho} in some detail.
Introduce for a continuous function $f: [0,1] \to \R$ and
$r < \rho < r + 1$, where $r$ is a nonnegative integer, the sequences
\be
 a(n) = 1,\quad b(n) = f(n^\rho \pmod{1}).
\ee
Denote by $J$ the associated Jacobi operator.
For $\alpha \in [0,1] \backslash \mathbb{Q}$, introduce the
skew-shift $T_{\alpha}: [0,1]^r \to [0,1]^r$ by
\be
 (T_{\alpha} \omega)_k = \begin{cases} \omega_0 + \alpha & k = 0\\ \omega_k + \omega_{k-1} & 1 \leq k \leq r -1.\end{cases}
\ee
Similarly as in the last part of Section~\ref{sec:ergodic}
we let $\beta_{\alpha}$ be the measure on $\mathcal{J}$ given by
the pushforward of the Lebesgue measure on $[0,1]^r$ under
\be
 [0,1]^r \ni \omega \mapsto \{1, f((T^n \omega)_r)\}_{n\in\Z} \in \mathcal{J}.
\ee

\begin{lemma}
 We have that
 \be
  \lim_{N\to\infty} \frac{1}{N} \sum_{n=0}^{N-1} \delta_{J^{(n)}} = \int_{0}^{1} \beta_{\alpha} d\alpha
 \ee
 in the weak $\ast$ topology.
\end{lemma}

\begin{proof}
 By Lemma~\ref{lem:compsupp} it suffices to check convergence for
 compactly supported functions $g$. Since, every such $g$ will be a continuous
 function of $\{b(n)\}_{n = -K}^{K}$ for some $K \geq 1$ it suffices
 to check that these converge. Since, $f$ is also continuous.
 It suffices to show that
 $$
  \{(n + j)^\rho\}_{j=-K}^{K}
 $$
 has the same distribution in $[0,1)^{2K + 1}$ as the orbits of
 the skew-shifts would as $n \to \infty$.
 
 Furthermore, by Lemma~2.1 in \cite{k1} and an easy argument we see
 that both $(n + j)^\rho$ and $(T^{n + j} \omega)_r$ are essentially
 given by degree $r$ polynomials, and thus uniquely determined by
 $$
  \{(n + j)^\rho\}_{j=0}^{r}\quad\text{and}\quad\{(T^{n + j} \omega)_r\}_{j=0}^{r}.
 $$
 Now Lemma~2.3 in \cite{k1} implies that the coefficients of the
 first polynomial are uniformly distributed, and a quick computation
 shows the same for the skew-shift, finishing the proof.
\end{proof}

This lemma combined with Theorem~\ref{thm:main} implies

\begin{corollary}\label{cor:lyaprho}
 For almost every $E$, we have that
 \be
  L(E) = \int_{0}^{1} \gamma_{\beta_{\alpha}}(E) d\alpha.
 \ee
\end{corollary}

This corollary resolves the discrepancy between the numerical
and perturbation theoretical computations in \cite{gf} and shows
in particular that the Lyapunov exponent only depends on the
integer part of $\rho$.

In particular, in the case of $r = 0$, that is $0 < \rho < 1$,
one can compute that $\gamma_{\beta_{\alpha}}(E) = 0$ for
exactly
\be
 E \in [-2 + f(\alpha), 2 + f(\alpha)].
\ee
Hence, we see that $L(E) = 0$ for
\be
 E \in [-2 + \max(f), 2 + \min(f)],
\ee
which was first observed by Simon and Zhu in \cite{sz} for continuum Schr\"odinger
operators. We observe further spectral properties in the following result.

\begin{theorem}
 We have that
 \begin{enumerate}
  \item \textbf{Stolz:} If $f$ extends to a smooth function on the circle, then $J$
   has purely absolutely continuous spectrum in $[-2+\max(f), 2+\min(f)]$.
  \item \textbf{Remling's Oracle:} If $f(0) \neq f(1)$, then the absolutely continuous spectrum
   of $J$ is empty.
 \end{enumerate}
\end{theorem}

\begin{proof}
 Part (i) is \cite{s}. Part (ii) follows from Remling's Oracle Theorem,
 which is found in \cite{rem}.
\end{proof}

%
%
%
%

\section{The integrated density of states and the Lyapunov exponent}
\label{sec:lyap}

In this section, we will proof Theorem~\ref{thm:main}. For this
we have to discuss some further properties of the Lyapunov exponent.

We will now assume that we are given a fixed Jacobi operator
$J:\ell^2(\Z) \to \ell^2(\Z)$. We denote by $J_+$ its restriction
to $\ell^2(\Z_+)$, $\Z_+ = \{0,1,2,3,\dots\}$, and by $J_{\Lambda}$
its restriction to $\ell^2(\Lambda)$ for $\Lambda\subseteq \Z$
an interval. Denote by $E_j(\Lambda)$ an increasing enumeration of
the eigenvalues of $J_{\Lambda}$, introduce the density of
states measure $\nu_n$ of $J_{[0,n-1]}$ by
\be\label{eq:defnun}
 \nu_n = \frac{1}{n} \sum_{j=0}^{n-1} \delta_{E_j([0,n-1])}.
\ee
For $\im(z) > 0$, introduce the Weyl--Titchmarsh $m$ function
\be
 m_+(z,J) = \spr{\delta_0}{(J_+ - z)^{-1} \delta_0}.
\ee
We will show the following theorem, which is essential in the proof
of Theorem~\ref{thm:main}. Similar results can be found in
Poltoratski--Remling \cite{pr}.

\begin{theorem}\label{thm:existlyap}
 Given a sequence $N_j \to \infty$, assume that
 \be\label{eq:muconv}
  \lim_{j\to\infty} \frac{1}{N_j} \sum_{n=0}^{N_j-1} \delta_{J^{(n)}} = \mu.
 \ee
 in the weak $\ast$ topology on $\mathcal{M}^1$. Then
 \begin{enumerate}
  \item The integrated density of states measures converge in the
   weak $\ast$ topology
   \be
    \lim_{j\to\infty} \nu_{N_j} = \nu.
   \ee
  \item For $\im(z)> 0$ and $A$ as defined in \eqref{eq:defA}, we have that
   \be\label{eq:Ltointmp}
    \ol{L}(z, \{N_j\}) = \log(A^{-1}) - \int_{\mathcal{J}} \log|m_+(z,J)| d\mu(J).
   \ee
  \item For almost every $E \in \R$, we have that
   \be\label{eq:Lepsto0}
    \ol{L}(E, \{N_j\}) = \lim_{\eps\to 0} \ol{L}(E + i \eps, \{N_j\}).
   \ee
 \end{enumerate}
\end{theorem}

Since the map $(a,b) \mapsto \log(a(0))$ is continuous, we see that
\eqref{eq:muconv} implies
\be\label{eq:defA}
 A := \exp\left(\int_{\mathcal{J}} \log(a(0)) d\mu(J)\right)
  = \lim_{j \to \infty} \exp\left(\frac{1}{N_j} \sum_{n=0}^{N_j+1} \log(a(n))\right),
\ee
where $A$ is the constant in \eqref{eq:Ltointmp}. We are now ready for

\begin{proof}[Proof of the Theorem~\ref{thm:main}]
 We first observe that
 $$
  \log(A^{-1}) = \int_{\mathcal{E}} \log(A_{\beta}^{-1}) d\alpha(\beta).
 $$
 We may compute for almost every $E$, that
 \begin{align*}
  L(E) & = \lim_{\eps \to 0} L(E + i \eps) & \text{by }\eqref{eq:Lepsto0} \\
   & = \lim_{\eps\to 0} \left(\log(A^{-1}) -\int_{\mathcal{J}} \log|m_+(E + i \eps,J)| d\mu(J)\right) & \text{by } \eqref{eq:Ltointmp}\\
   & = \lim_{\eps\to 0} \left(\int_{\mathcal{E}}\log(A_{\beta}^{-1}) - \left(\int_{\mathcal{J}} \log|m_+(E + i \eps,J)| d\beta(J) \right) d\alpha(\beta) \right)
    & \text{by } \eqref{eq:edc} \\
   & = \lim_{\eps\to 0} \left(\int_{\mathcal{E}} \gamma_{\beta}(E + i \eps) d\alpha(\beta) \right) &\text{by } \eqref{eq:egamma} \\
   & = \int_{\mathcal{E}} \gamma_{\beta}(E) d\alpha(\beta) &\text{by } \eqref{eq:limegamma} &.
 \end{align*}
 This implies the first claim.
 The second claim follows by Thouless' formula.
\end{proof}

We now proceed to prove Theorem~\ref{thm:existlyap}.
Introduce by $s$ and $c$ the sine and cosine
solution of $J$ (as a formal difference equation), satisfying the
initial conditions
\be
 \begin{pmatrix} c(z,0) & s(z,0) \\ c(z,-1) & s(z,-1) \end{pmatrix} =
 \begin{pmatrix} 1 & 0 \\ 0 & 1 \end{pmatrix}.
\ee
We observe that, we have that
\be\label{eq:cnsn}
 \begin{pmatrix} c(z,n) & s(z,n) \\ a(n-1) c(z,n-1) & a(n-1) s(z,n-1) \end{pmatrix} =
 \prod_{j=n-1}^{0} \frac{1}{a(j)} \begin{pmatrix} z - b(j) & 1 \\ a(j)^2 & 0 \end{pmatrix}.
\ee
We note that
\be\label{eq:detformula}
 c(z,n) = \frac{\det(z - J_{[0,n-1]})}{\prod_{j=0}^{n-1} a(j)},\quad s(z,n) = \frac{\det(z - J_{[1,n-1]})}{\prod_{j=0}^{n-1} a(j)}.
\ee
For $\im(z) > 0$, we denote by $u_+(z,n)$ the solution of
\be
 H u_+ = zu_+,\quad u_+ \in \ell^2(\Z_+),\quad u_+(z,-1) = 1.
\ee
We then have that $u_+(z,0) = - a(0) m_+(z, J)$. 
We obtain that
\be\label{eq:upexpan}
 u_+(z,N) = (-1)^N \prod_{n=0}^{N} a(n) m_+(z, J^{(n)}).
\ee
Hence, we obtain that

\begin{lemma}\label{lem:up}
 Assume \eqref{eq:muconv}, then for $\im(z) > 0$
 \begin{align}\label{eq:olLequp}
  \ol{L}(z,\{N_j\}) &= - \lim_{j\to\infty} \frac{1}{N_j} \log|u_+(z,N_j)| \\
  \nn& = - \lim_{j\to\infty} \frac{1}{N_j} \log\sqrt{|u_+(z,N_j)|^2 + |u_+(z,N_j+1)|^2}\\
  \nn &= \log(A^{-1}) - \int_{\mathcal{J}} \log|m_+(z,(a,b))| d\mu(a,b).
 \end{align}
\end{lemma}

\begin{proof}
 Define $L_+(z,\{N_j\})$ as the limit
 $$
  L_+(z,\{N_j\}) = - \lim_{j \to \infty} \frac{1}{N_j} \sum_{n=0}^{N_j-1} \log|a(n) m_+(z, J^{(n)})|,
 $$
 which exists by continuity of $J \mapsto \log|m_+(z,J)|$, \eqref{eq:muconv}, and \eqref{eq:defA}.
 By \eqref{eq:upexpan}, $|m_+(z,J)| \leq \frac{1}{|\im(z)|}$,
 $$
  \sqrt{|u_+(z,N_j)|^2 + |u_+(z,N_j+1)|^2} = |u_+(z,N_j)| \sqrt{1 + |a(N_j+1) m_+^{(N_j+1)}(z)|}.
 $$
 we see that $L_+(z,\{N_j\})$ is equal to the all the quantities on the right hand
 side of \eqref{eq:olLequp}. In order to see the remaining inequality, observe that
 $$
  \begin{pmatrix} u_+(z,N_j) \\ a(N_j-1) u_+(z,N_j-1) \end{pmatrix} =
  \prod_{n=N_j-1}^{0} \frac{1}{a(n)} \begin{pmatrix} z - b(n) & 1 \\ a(n)^2 & 0 \end{pmatrix}
  \begin{pmatrix} m_+(z,J) \\ 1 \end{pmatrix},
 $$
 which implies the claim by the Ruelle--Osceledec theorem.
\end{proof} 

This shows (ii) of Theorem~\ref{thm:existlyap}.
Or next goal is to relate the Lyapunov exponent to the
asymptotics of the cosine solution. 
Introduce for $\im(z) > 0$
\be
 m_N(z) = \spr{\delta_N}{(J_{[0,N]} - z)^{-1} \delta_N}.
\ee
We have that

\begin{lemma}
 For $\im(z) > 0$,
 \be
  \frac{c(z,N+1)}{c(z,N)} = \frac{1}{a(N) m_N(z)},
 \ee
 and $|m_N(z)| \leq \frac{1}{|\im(z)|}$.
\end{lemma}

\begin{proof}
 Introduce $v$ for $0 \leq n \leq N+1$ by
 $$
  v(n) = \begin{cases} 0 & n < 0 \\ \spr{\delta_n}{(J_{[0,N]} - z)^{-1} \delta_N} & 0 \leq n \leq N \\ \frac{1}{a(N)} & n = N + 1. \end{cases}
 $$
 One then checks that $(J - z)v (n) = 0$ for $1 \leq n \leq N$
 and $v(-1) = 0$. Thus $c(z,n) = C \cdot v(n)$ for $0 \leq n \leq N+ 1$ for a fixed
 constant $C$. Now a computation shows that
 $$
  \frac{c(z,N+1)}{c(z,N)} = \frac{v(N+1)}{v(N)} = \frac{1}{a(N) m_N(z)}, 
 $$
 which shows the claim.
\end{proof}

As in the proof of Lemma~\ref{lem:up}, one can now show that
$$
 \lim_{j\to\infty} \frac{1}{N_j} \log|c(z,N_j)| = \lim_{j\to\infty} \frac{1}{N_j} \log\sqrt{|c(z,N_j)|^2|+c(z,N_j)|^2},
$$
which implies for $\im(z) > 0$ by the Ruelle--Osceledec theorem
\be\label{eq:upeqc}
 \lim_{j\to\infty} \frac{1}{N_j} \log|u_+(z,N_j)| = - \lim_{j\to\infty} \frac{1}{N_j} \log|c(z,N_j)|,
\ee
since the cosine solution $c$ can never decay, since $J$ is self-adjoint.
In particular, the limit on the right hand side of \eqref{eq:upeqc} exists for every $z$
with $\im(z) > 0$. 

\begin{lemma}
 We have that
 \be\label{eq:logczn}
  \frac{1}{n} \log|c(z,n)|
  = \int\log|z - t| d\nu_n - \frac{1}{n} \sum_{j=0}^{n-1} \log|a(j)|.
 \ee
\end{lemma}

\begin{proof}
 This is a consequence of \eqref{eq:detformula} and \eqref{eq:defnun}.
\end{proof}

\begin{lemma}
 Assume \eqref{eq:muconv}, then
 \be\label{eq:nulim}
  \nu = \lim_{j\to\infty} \nu_{N_j}
 \ee
 exists, and for $\im(z) > 0$
 \be\label{eq:thouless3}
  \lim_{j\to\infty} \frac{1}{N_j} \log|c(z,N_j)| = \log(A^{-1}) + \int \log|t - z| d\nu.
 \ee
 Furthermore, \eqref{eq:thouless3} even holds for almost every $z \in \R$.
\end{lemma}

\begin{proof}
 \eqref{eq:nulim} follows from \eqref{eq:logczn} and the fact that the family of
 functions $t \mapsto \log|t - z|$ for $\im(z) > 0$ separates points
 on the real axis. For the last statement, observe that \eqref{eq:logczn}
 remains valid for $z \in \R$, and then use Theorem~A.7. in \cite{siem}.
\end{proof}

This shows (i) of Theorem~\ref{thm:existlyap}.
Next, we observe that

\begin{lemma}
 For every $E \in \R$, we have that
 \be
  \ol{L}(E, \{N_j\}) \leq \log(A^{-1}) + \int \log|t - E| d\nu
 \ee
\end{lemma}

\begin{proof}
 First observe that $\ol{L}(z, \{N_j\})$ is a submean function of $z$,
 and $z \mapsto \log(A^{-1}) + \int \log|t - z| d\nu$ is subharmonic.
 This implies the claim by Theorem~1.1. in \cite{cs2}.
\end{proof}

We now come to

\begin{proof}[Proof of Theorem~\ref{thm:existlyap} (iii)]
 This is a consequence of the last lemma, and the fact that
 $$
  |c(E, N_j)| \leq \left\|\prod_{n=N_j-1}^{0} \frac{1}{a(n)} \begin{pmatrix} z - b(n) & 1 \\ a(n)^2 & 0 \end{pmatrix}\right\|
 $$
 by \eqref{eq:cnsn}.
\end{proof}

\section*{Acknowledgments}

This project has profited from many discussions
with Jon Chaika, David Damanik, and Daniel Lenz. Furthermore,
I wish to thank Barry Simon for useful discussion and for
the invitation to the \textit{27th Western States Meeting}, where
the ideas related to the Denisov--Rakhmanov--Remling
theorem were born. Finally, I wish to thank the organizers
of the workshop \textit{Random Schr\"odinger Operators: Universal Localization, Correlations, and Interactions},
at the \textit{Banff International Research Station} for
their invitation, since this project took its final
form during this workshop.


\begin{thebibliography}{xxx}

 \bibitem{cdk} J. Chaik, D. Damanik, H. Kr\"uger, \textit{Schr\"odinger Operators defined by Interval Exchange Transformations}, 
   J. Mod. Dyn. \textbf{3:2} (2009).

 \bibitem{cs2} W. Craig, B. Simon, \textit{Subharmonicity of the Lyaponov index},
   Duke Math. J. \textbf{50-2} (1983), 551--560.

 \bibitem{d1} D. Damanik, \textit{Lyapunov exponents and spectral analysis of ergodic Schr\"odinger operators:
   a survey of Kotani theory and its applications},  Spectral theory and mathematical physics: a Festschrift
   in honor of Barry Simon's 60th birthday,  539--563, Proc. Sympos. Pure Math., \textbf{76-2}, Amer. Math. Soc., Providence, RI, 2007.

 \bibitem{den} S. Denisov, \textit{On Rakhmanov's theorem for Jacobi matrices},
   Proc. Amer. Math. Soc. \textbf{132:3}, 847--852 (2004).

 \bibitem{els} J. Elstrodt,  \textit{Ma\ss- und Integrationstheorie}, (German) [Measure and integration theory] Fourth edition. Springer-Lehrbuch. [Springer Textbook] Grundwissen Mathematik. [Basic Knowledge in Mathematics] Springer-Verlag, Berlin, 2005. xvi+434 pp. ISBN: 3-540-21390-2 

 \bibitem{gf} M. Griniasty, S. Fishman, \textit{Localization by pseudorandom potentials in one dimension}, 
   Phys. Rev. Lett. \textbf{60}, 1334--1337 (1988)

 \bibitem{kh} A. Katok, B. Hasselblatt,  \textit{Introduction to the modern theory of dynamical systems},
   With a supplementary chapter by Katok and Leonardo Mendoza. 
   Encyclopedia of Mathematics and its Applications, \textbf{54}. Cambridge University Press, Cambridge, 1995. xviii+802 pp.

 \bibitem{k1} H. Kr\"uger, \textit{A family of Schr\"odinger operators whose spectrum is an interval},
   Comm. Math. Phys. (to appear).

   
 \bibitem{phe} R. Phelps, \textit{Lectures on Choquet's theorem}, Second edition,
   Lecture Notes in Mathematics, \textbf{1757}. Springer-Verlag, Berlin, 2001. viii+124 pp. ISBN: 3-540-41834-2 

 \bibitem{pr} A. Poltoratski, C. Remling, \textit{Reflectionless Herglotz functions and generalized Lyapunov exponents},
   preprint.

 \bibitem{rak} E. A. Rakhmanov, \textit{On the asymptotics of the ratio of orthogonal polynomials. II},
   Math. USSR Sb. \textbf{46} (1983), 105--117. 

 \bibitem{rem} C. Remling, \textit{The absolutely continuous spectrum of Jacobi matrices},
   preprint.

 \bibitem{siem} B. Simon, \textit{Equlibrium Measures and Capacities in Spectral Theory},
   Inverse Problems and Imaging \textbf{1} (2007), 713-772.


 \bibitem{sz} B. Simon, Y.F. Zhu, \textit{The Lyapunov exponents for Schr\"odinger operators with slowly oscillating potentials},
   J. Funct. Anal. \textbf{140} (1996), 541-556

   
 
 \bibitem{s} G. Stolz, \textit{Spectral theory for slowly oscillating potentials. I. Jacobi matrices},
   Manuscripta Math.  \textbf{84},  no. 3-4, 245--260 (1994).


 \bibitem{sur} S. Surace, \textit{Positive Lyapunov Exponents for a Class of Ergodic Schr\"odinger Operators},
   Comm. Math. Phys.   \textbf{162} (1994),  529--537.

 \bibitem{tjac} G. Teschl, {\em Jacobi Operators and Completely Integrable Nonlinear Lattices}, Math. Surv. and Mon. {\bf 72},
   Amer. Math. Soc., Rhode Island, 2000.


\end{thebibliography}
\end{document}